\documentclass[11pt]{amsart}

\usepackage{latexsym,fullpage}
\usepackage{tikz}
\usetikzlibrary{arrows}
\usepackage{pstricks-add}
\usepackage{amssymb}
\usepackage{amsfonts}
\usepackage{amsmath}
\usepackage{amsthm}
\usepackage{graphics,pstricks}
\usetikzlibrary{arrows}
\DeclareSymbolFont{legacymaths}{OT1}{cmr}{m}{n}
\SetSymbolFont{legacymaths}{bold}{OT1}{cmr}{bx}{n}
\usepackage[english]{babel}
\usepackage{times}
\usepackage[T1]{fontenc}
\usepackage{color}
\usepackage[colorlinks=true, linkcolor=blue, anchorcolor=blue, citecolor=blue, filecolor=blue, menucolor= blue, urlcolor=red]{hyperref}

\newtheorem{thm}{Theorem}[section]
\newtheorem{lem}[thm]{Lemma}
\newtheorem{cor}[thm]{Corollary}
\newtheorem{pro}[thm]{Proposition}

\theoremstyle{definition}
\newtheorem{defn}[thm]{Definition}
\newtheorem{defn-rem}[thm]{Definition and Remark}
\newtheorem{exmp}[thm]{Example}

\newtheorem{rem}[thm]{Remark}

\newtheorem{nota}[thm]{Notation}

\numberwithin{equation}{section}

\def\ni{\noindent}

\def\X{{\mathbb X}}
\def\L{{\mathbb L}}

\def\H{{\text {\bf H}}}

\def\Y{{\mathbb Y}}
\def\P{{\mathbb P}}

\def\L{{\mathbb L}}

\def\ds{\displaystyle}

\def\be{\bfseries\em }

\def\k{\Bbbk}

\def\sat{{\rm sat}}

\begin{document}

\title{Green's theorem and Gorenstein sequences}

\author[J. Ahn] {Jeaman Ahn}
\address{Department of Mathematics Education, Kongju National University, 182, Shinkwan-dong, Kongju, Chungnam 314-701, Republic of Korea}
\email{jeamanahn@kongju.ac.kr}

\author{Juan C. Migliore} 
\address{Department of Mathematics, 
 University of Notre Dame, 
  Notre Dame,
IN 46556 \\
 USA}
 \email{migliore.1@nd.edu}

\author[Y.S. SHIN]{Yong-Su Shin}
\address{Department of Mathematics, Sungshin Women's University, Seoul, Republic of Korea, 136-742}
\email{ysshin@sungshin.ac.kr }


\begin{abstract} 
We study consequences, for a standard graded algebra, of extremal behavior in Green's Hyperplane Restriction Theorem. First, we extend his Theorem 4 from the case of a plane curve to the case of a hypersurface in a linear space. Second, assuming a certain Lefschetz condition, we give a connection to extremal behavior in Macaulay's theorem. We apply these results to show that $(1,19,17,19,1)$ is not a Gorenstein sequence, and as a result we classify the sequences of the form $(1,a,a-2,a,1)$ that are Gorenstein sequences.

\end{abstract}
\maketitle
\tableofcontents

\section{Introduction} 

In the study of Hilbert functions of standard graded algebras, Macaulay's theorem \cite{M} and Green's theorem \cite{G1} stand out as being of fundamental importance both on a theoretical level and from the point of view of applications. 
Macaulay's theorem regulates the possible growth of the Hilbert function from one degree to the next. It is a stunning fact that strong geometric consequences arise whenever the maximum possible  growth allowed by this theorem is achieved \cite{Go}, \cite{BGM}, \cite{AM}, or even when the maximum is {\em almost} achieved \cite{CM}. Green's theorem regulates the possible Hilbert functions of the restriction modulo a general linear form. It is a less-studied question to ask what  happens if the maximum possible Hilbert function occurs for this restriction, although already Green gave some intriguing results \cite{G1}, \cite{BZ} in his so-called ``Theorem~3" and ``Theorem 4," and some results in this direction can also be found in \cite{AGS}. To our knowledge, the connections between these two kinds of extremal behavior have not previously been studied.

One area where both Macaulay's theorem and Green's theorem have been applied very profitably is the problem of classifying the  Hilbert functions of Artinian Gorenstein algebras (i.e. of finding all possible {\em Gorenstein sequences}). Of course this problem is probably intractable in full generality. However, many papers have been written on the subject, and we cannot begin to list them all here. Even the special case of socle degree 4 (i.e. Gorenstein sequences of the form $(1,a,b,a,1)$) has been carefully studied (see for instance \cite{St}, \cite{MNZ:3},  \cite{AS:1}, \cite{BZ}, \cite{MZ:2}, \cite{AS:3}), 
but a full classification remains open.

If $b \geq a$, these sequences are completely understood (see for instance  \cite{Ha:3}). It is the non-unimodal case that is of great interest. The study was begun by Stanley \cite{St}, who showed that $(1,13,12,13,1)$ is a Gorenstein sequence, so it follows easily that non-unimodal examples exist for all $a \geq 13$.  In \cite{MZ:2} the authors showed that Stanley's example is the smallest possible, i.e. that if $a \leq 12$ then $b \geq a$. This leads to an easy classification of the possibilities when $a-b = 1$. 
There remains the question of ``how non-unimodal can the Hilbert function be?" Stanley conjectured an asymptotic lower bound for $b$ as $a \rightarrow \infty$ in \cite{St2}, which was proved (including sharpness) in \cite{MNZ:3}. However, it is not known for any fixed value of $a \geq 18$ exactly which are the possible values of $b$, although F. Zanello has pointed out to us that for some specific values of $a$ (e.g. $a = 24$) it is fairly easy to find all possible $b$.

In this paper we make progress on both problems. First, we study some consequences of extremality for Green's theorem, including an analysis of a situation where we have an equivalence between this extremality and that for Macaulay's theorem. Next we apply this work to produce new results on Gorenstein sequences of socle degree 4.

More precisely, after recalling known facts in section \ref{background section}, our main goal in section \ref{mac-green} is to find new consequences of extremal behavior in Green's theorem. We recall Green's Theorem 4 and we first prove a direct generalization in Theorem \ref{T:20160707-213}, passing from Green's case of a plane curve to the case of a hypersurface in a linear subspace. Our main result in this section is Theorem \ref{relate}, which gives a connection, under certain assumptions, between extremal behavior for Green's theorem and extremal behavior for Macaulay's theorem. Because of this connection, Gotzmann's theorem applies as it did in the paper \cite{BGM} to give strong geometric consequences, which we explore in Corollary~\ref{cor of thm}. We also show that Green's theorem is ``sequentially sharp" in Corollary \ref{C:20160706-215}.

We apply our new results on Green's theorem in Section \ref{GorSection} to show that the sequence $(1,19,17,19,1)$ is not Gorenstein (Theorem \ref{T:20160707-301}). Our proof brings together a number of different techniques. The result is the main ingredient for our Corollary \ref{r-2}, which completes the classification of the socle degree 4 Gorenstein sequences with $a-b = 2$ (with the notation introduced above) by proving that the sequence is Gorenstein if and only if $a \geq 20$.

Theorem \ref{relate} makes a certain numerical assumption as well as a certain Lefschetz assumption in order to conclude that the two different kinds of extremal behavior are equivalent. This gives a new illustration of the importance of the so-called Lefschetz properties, which have been studied very extensively in the last two decades, especially the Weak Lefschetz Property (WLP) and the Strong Lefschetz property (SLP). 
However, it is worth noting here that our Lefschetz assumption is much milder than  WLP. Instead, we only assume that multiplication on our algebra by a general linear form is injective in just one degree. Interestingly, there are two different degrees where such an assumption leads to the equivalence mentioned above.
This Lefschetz (injectivity) assumption can be phrased in more than one way, as shown in Lemma \ref{WLP cond}.
It also leads to a surprisingly simple but useful result, Lemma \ref{socle}, which forces the existence of a socle element in a specific degree. It is a small improvement of \cite[Proposition 2.1 (b)]{MMN}, although our proof is completely different. It provides a very simple way to rule out cases, via the existence of socle elements, in our study of Gorenstein sequences in the last section.

Finally, we make a remark on the characteristic. In their paper \cite{BZ}, M. Boij and F. Zanello (and M. Green in the appendix) make a careful study of its role. They note that Green's theorem and Macaulay's theorem are true independently of the characteristic. However, Green's Theorem 3 (see Corollary \ref{Green 3} below) requires $\hbox{char } \k \neq 2$, and Green's Theorem 4 (see Theorem \ref{T:20160606-211}) requires $\hbox{char } \k = 0$ (although they point out that the characteristic can simply be ``large enough" in a sense that they make precise). Since our Theorem \ref{T:20160707-213}  uses Green's Theorem 4 for the induction, we also assume characteristic zero there, and hence the same is true of Corollary \ref{Green 3}. And because we use this result in one place in the proof of Theorem \ref{T:20160707-301}, we also assume it there. However, the main results of Section \ref{mac-green} are independent of the characteristic.

\section{Background} \label{background section}

Let $R = \k[x_0,\dots,x_n]$ be the homogeneous polynomial ring and let $A = R/I$ be a standard graded Artinian $\k$-algebra, where $\k$ is an infinite field. The {\em Hilbert function} of $A$ is the function on the natural numbers defined by $ {\bf H}(A,d) = \dim_\k [A]_i$. Since $A$ is Artinian, we often represent this function by the {\em $h$-vector} $(1 = h_0, h_1, \dots, h_e)$ with $h_e > 0$, where $h_i = {\bf H}(A,i)$. The integer $e$ is called the {\em socle degree} of~$A$.

Let $L \notin I$ be a linear form in $R$. We have the  graded exact sequence 
\begin{equation} \label{ses}
0 \rightarrow R/(I:L)(-1) \rightarrow R/I \rightarrow R/(I,L) \rightarrow 0.
\end{equation}

\begin{nota} \label{notation1}
Throughout this paper we shall adopt the following:
\[
\begin{array}{rcl}
h_i & = & \dim_\k [A]_i \\
b_i & = & \dim_\k [R/(I:L)]_i   \\
\ell_i & = & \dim_\k [R/(I,L)]_i .
\end{array}
\]
\end{nota}

\smallskip

The following is well known, and the first part follows from the above sequence.

\begin{lem} \label{L:202} Let $A=R/I$ be a graded Artinian  algebra, 
and let $L\notin I$ be a linear form of $R$. Then we have
\[
\H:=(h_0,h_1,\dots,h_e)=(1,b_0+\ell_1,\dots,b_{e-2}+\ell_{e-1},b_{e-1} + \ell_e) .
\]
Furthermore, if $A$ is Gorenstein then so is $R/(I:L)$, and $b_{e-1} = h_e = 1$:
\[
\H:=(h_0,h_1,\dots, h_{e-1},h_e=1)=(1,b_0+\ell_1,\dots,b_{e-2}+\ell_{e-1},b_{e-1} =1) 
\]
\end{lem} 

In this paper $A$ will always be Gorenstein, and we will  often use the following notation.

\begin{nota} \label{decomp notation}
With notation as in Lemma~\ref{L:202}, we shall simply call the following diagram
\[
\begin{array}{ccccccccccccccccc}
h_0 & h_1 & h_2 & \cdots & h_{e-1} &h_e\\
       & b_0 & b_1 & \cdots & b_{e-2} & b_{e-1}\\ \hline 
\ell_0 & \ell_1 & \ell_2 & \cdots & \ell_{e-1} & \ell_e
\end{array}
\]
{\em the decomposition of the Hilbert function $\H$}. 
\end{nota}


\begin{defn} \label{binomial exp} Let $r$ and $i$ be positive integers. The {\em $i$-binomial expansion of $r$} is 
\[
r_{(i)} = \binom{r_i}{ i}+\binom{r_{i-1}}{ i-1}+...+\binom{r_j}{ j},
\]
 where $r_i>r_{i-1}>...>r_j\geq j\geq 1$. Such an expansion always exists and is unique (see, e.g.,  \cite{BH}, Lemma 4.2.6). Following \cite{BH}, we define, for any integers $a$ and $b$,
$$
{r_{(i)}}|_{a}^{b}=\binom{r_i+b}{ i+a}+\binom{r_{i-1}+b}{ i-1+a}+...+\binom{r_j+b}{ j+a},
$$
where we set $\binom{m}{ c}=0$ whenever $m<c$ or $c<0$.
\end{defn}

\begin{thm}[\cite{G1}, \cite{M}]\label{T:201} 
Let $h_d$ be the entry of degree $d$ of the 
Hilbert function of $R/I$ and let $\ell_d$ be the degree $d$ entry  of the Hilbert 
function of $R/(I,L)$, where $L$ is a general linear form of $R$. Then, we have the following inequalities.

\begin{enumerate}
\item [(a)] {\rm Macaulay's Theorem:} $h_{d+1} \leq \left((h_d)_{(d)}\right)|^{+1}_{+1}$.
\item [(b)] {\rm Green's Hyperplane Restriction Theorem (Theorem 1):} $\ell_d \leq \left((h_{d})_{(d)}\right)|^{-1}_{\phantom{-}0}.$
\end{enumerate}
\end{thm}




\begin{thm}[\cite{Go}, Gotzmann's Persistence Theorem] \label{T:20110413-208} Let $I$ be a homogeneous ideal generated in degrees $\le d+1$. If in Macaulay's estimate,
$$
h_{d+1}=((h_{d})_{(d)})|^{+1}_{+1},
$$
then $I$ is $d$-regular and
$$
h_{t+1}=((h_{t})_{(t)})|^{+1}_{+1}
$$
for all $t\ge d$. 
\end{thm}


The following result, with a small change in notation, is \cite{Za4}, Theorem 3.5.

\begin{lem}\label{L:20160525-205}
Let $\H=(h_0,\dots,h_{d-1}, h_d, h_{d+1},\dots,h_e)$ be an $h$-vector of an Artinian ring $R/J$. Suppose that, for some $d>0$, there is a positive integer  $\varepsilon>0$ such that 
$$
h_{d-1}=(h_{d})_{(d)}|^{-1}_{-1}+\varepsilon \text{ \ \ and \ \ } h_{d+1}=(h_{d})_{(d)}|^{+1}_{+1}.
$$ 
Then the ring $R/J$ has socle of dimension $\varepsilon$ in degree $d-1$. Consequently, if $R/J$ has the graded minimal free resolution $\mathbb F$, as above, then
$$ 
\beta_{r, r+d-1}(R/J)=\varepsilon.
$$
\end{lem}

%

\begin{rem}\label{R:20160525-206}
As noted in \cite{Za4} Example 3.6, Lemma~\ref{L:20160525-205} slightly generalizes Theorem~3.4 in \cite{GHMS} (see also \cite{CI}). For example, consider an Artinian ring $R/I$ with an $h$-vector
$$
\H=(1,18,16,18,28).
$$  
Note that there is no positive integer $h_2$ such that
$$
(h_2)_{(2)}|^{+1}_{+1}=h_3,
$$
and thus we cannot apply Theorem~3.4 in \cite{GHMS} to show that $R/J$ has a socle element in degree $3$.

However, by Theorem \ref{T:201} (a), $R/I$ has maximal growth in degree $3$. Moreover, since
$$
(18)_{(3)}|^{-1}_{-1}=10,
$$
we get that $R/I$ has a $6$-dimensional socle  in degree $3$. 
\end{rem}


\begin{pro}[{\cite{BS}}]\label{P:20160530-301}
$I$ is $m$-saturated if and only if, for a general linear form $L\in R$, $(I:L)_d=I_d$ for every $d\ge m$.
\end{pro}

The following result is well known and follows from standard methods.

\begin{thm} \label{wellknown}
If $(1,n,a,n,1)$ is a Gorenstein $h$-vector then so are $(1,n,b,n,1)$ for each $a \leq b \leq \binom{r+1}{2}$ and $(1,n+1,a+1,n+1,1)$. 
\end{thm}

%



\section{Relations between Green's theorem and Macaulay's theorem} \label{mac-green}

Several papers have studied geometric and algebraic consequences for a standard graded algebra when its Hilbert function achieves the maximal growth in some degree allowed by Macaulay's theorem (Theorem \ref{T:201}  (a)).  See for instance \cite{Go}, \cite{davis}  \cite{BGM},  \cite{AM}, \cite{CM}. Not as much work has been done, to our knowledge, exploring the consequences of extremal behavior of the Hilbert function under Green's theorem (Theorem \ref{T:201} (b)) other than Green's Theorem 3 and Theorem 4 (see \cite{G1}, \cite{BZ}).
In this section we generalize Green's Theorem 4, and we give some results that connect the two kinds of extremality.

Throughout this section and the next we will use binomial expansions, and we refer to Definition \ref{binomial exp} for the conditions on the various integers. We first recall Green's Theorem 4. Recall also from Notation \ref{notation1} that $\ell_d = \dim_\k [R/(I,L)]_d$. Observe that the indicated restriction is extremal according to Green's theorem.

\begin{thm}[Green's Theorem 4] \label{T:20160606-211}
Let $I \subset R = \k[x_0,\dots,x_n]$ be a homogeneous ideal. Assume that $\hbox{char } \k = 0$ and suppose that, for some integers $m$ and $d$, $1 \leq m \leq d$, we have the binomial expansion
\[
h_d = md + 1 - \binom{m-1}{2} = \binom{d+1}{d} + \binom{d}{d-1} + \dots + \binom{d-(m-2)}{d-(m-1)}
\]
and $\ell_d = m =  (h_d)|^{-1}_{\phantom{-}0}$. Then in degree $d$, $I$ is the ideal of a plane curve of degree $m$. That is,
\[
[I]_d = \langle L_0,L_1,L_2,\dots,L_{n-3} \rangle + F \cdot [R]_{d-m}
\]
where $L_0,L_1,L_2,\dots,L_{n-3}$ are linearly independent linear forms and $F$ is a form of degree $m$.
\end{thm}

This result can be generalized as follows.

\begin{thm}\label{T:20160707-213} 
Let $I \subset R = \k[x_0,\dots,x_n]$ be a homogeneous ideal. Assume that $\hbox{char } \k = 0$ and that for some degree $d$ and integers $c$ and $k$ we have the binomial expansion
\[
h_d=\binom{d+c}{d}+\cdots+\binom{d+c-k}{d-k} \quad \text{ and } \quad \ell_d =(h_d)|^{-1}_{\phantom{-}0}.
\]
Then there is a hypersurface $F$ of degree $k+1$ in a $(c+1)$-dimensional linear space $\Lambda\subset \mathbb P(R_1)$ such that 
\[
[I]_d = [I_{\Lambda} +I_{F}]_d = \langle L_0,L_1,\dots,L_{n-c-2} , F \rangle_d .
\]

\end{thm}

\begin{proof}
The proof is by induction on $c$. Let $L$ and $L'$ be general linear forms. The case $c=1$ is Theorem~\ref{T:20160606-211} (Green's Theorem 4), so we will assume $c \geq 2$. 
Consider the  diagram 

\begin{equation} \label{diag}
\begin{array}{cccccccccccccc}
&& 0 && 0 && 0 \\
&& \downarrow && \downarrow && \downarrow \\
 && [R/((I:L):L') ]_{d-2} && [R/(I:L')]_{d-1} && [R/((I,L):L')]_{d-1} \\
&& \phantom{\scriptstyle \times L'} \downarrow \scriptstyle \times L' &&  \phantom{\scriptstyle \times L'} \downarrow \scriptstyle \times L'  &&  \phantom{\scriptstyle \times L'} \downarrow \scriptstyle \times L'  \\
0 & \rightarrow & [R/(I:L)]_{d-1} & \stackrel{\times L}{\longrightarrow} & [R/I]_{d} & \rightarrow & [R/(I,L)]_{d} & \rightarrow & 0 \\
&& \downarrow && \downarrow && \downarrow \\
&& [R/((I:L),L')]_{d-1} && [R/(I,L')]_{d} && [R/(I,L,L')]_{d} \\
&& \downarrow && \downarrow && \downarrow \\
&& 0 && 0 && 0 
\end{array}
\end{equation}
The assumptions give the dimensions in the middle row of (\ref{diag}) of the second and third vector spaces:

\begin{equation} \label{I}
\begin{array}{rcl}
h_d & = & \dim [R/I]_d \\[.5ex] 
& = & \displaystyle \binom{d+c}{d}+\cdots+\binom{d+c-k}{d-k} \\[2ex]  
& = & \displaystyle \binom{d+c}{c}+\cdots+\binom{d+c-k}{c} \\[2ex]  
& = & \displaystyle \binom{d+c+1}{c+1} - \binom{d+c-k}{c+1},
\end{array}
\end{equation}
and 
\begin{equation} \label{I,L}
\begin{array}{rcl}
\dim [R/(I,L)]_d 
& = & \displaystyle \binom{d+c-1}{d} + \dots + \binom{d+c-k-1}{d-k} \\[2ex]    
& = & \displaystyle \binom{d+c-1}{c-1} + \dots + \binom{d+c-k-1}{c-1} \\[2ex]   
& = & \displaystyle \binom{d+c}{c} - \binom{d+c-k-1}{c}.
\end{array} 
\end{equation}
Then a calculation gives
$$
\dim [R/(I:L)]_{d-1} = \binom{d+c-1}{d-1} + \cdots + \binom{d+c-k-1}{d-k-1} =  \binom{d+c}{c+1} - \binom{d+c-k-1}{c+1}.
$$
Looking at the first column of (\ref{diag}), Green's Theorem 1 then gives
$$
\dim [R/((I:L),L')]_{d-1} \leq \binom{d+c-2}{d-1} + \dots + \binom{d+c-k-2}{d-k-1} = \binom{d+c-1}{c} - \binom{d+c-k-2}{c}.
$$
Green's Theorem 1 applied to the third column of (\ref{diag}), thanks to our assumptions, gives
$$
\dim [R/(I,L,L')]_{d} \leq (h_{d})|^{-2}_{\phantom{-}0} =  \binom{d+c-1}{c-1} - \binom{d+c-k-2}{c-1}.
$$

Since
$$
((I,L'):L) \supseteq ((I:L),L'),
$$
we obtain
$$
\dim [R/((I,L'):L)]_{d-1} \leq  \binom{d+c-1}{c} - \binom{d+c-k-2}{c}
$$
and the sequence
$$
0 \rightarrow [R/((I,L'):L)]_{d-1} \rightarrow [R/(I,L')]_{d} \rightarrow [R/(I,L,L')]_{d} \rightarrow 0
$$
gives 
$$
\begin{array}{rcl}
\displaystyle \binom{d+c}{c} - \binom{d+c-k-1}{c} & = & \dim [R/(I,L')]_d \\  
& = & \dim [R/((I,L'):L)]_{d-1} + \dim [R/(I,L,L')]_{d} \\[.5ex]   
& \leq & \displaystyle \left [  \binom{d+c-1}{c} - \binom{d+c-k-2}{c}  \right ]+ \left [ \binom{d+c-1}{c-1} - \binom{d+c-k-2}{c-1} \right ] \\[2ex] 
& = & \displaystyle \binom{d+c}{c} - \binom{d+c-k-1}{c}.
\end{array}
$$
We conclude
$$
\dim [R/((I,L'):L)]_{d-1} =  \binom{d+c-1}{c} - \binom{d+c-k-2}{c}
$$
and
\begin{equation} \label{I,L,L'}
\begin{array}{rcl}
\dim [R/(I,L,L')]_{d} &  = &  \displaystyle  \binom{d+c-1}{c-1} - \binom{d+c-k-2}{c-1} \\ \\
& = & (h_d) \mid^{-2}_{\phantom{-}0}.
\end{array}
\end{equation}
Now combining (\ref{I,L}) and (\ref{I,L,L'}), we see that the ideal $(I,L)$ satisfies the inductive hypothesis for $c-1$. By induction, then, $(I,L)$ is the saturated ideal of some hypersurface, $F'$, of degree $k+1$ in a linear space $\Lambda'$ of dimension $c$, which is contained in the hyperplane defined by $L$. 

Let $Y$ be the scheme in $\mathbb P^{n}$ defined by $I$ (which a priori is not necessarily saturated in degree $d$). Then $F'$ is the hyperplane section of $Y$ cut out by the general hyperplane defined by $L$. Since $F'$ is arithmetically Cohen-Macaulay, $Y$ must be the union of a hypersurface of degree $k+1$ in some linear space $\Lambda$ of dimension $c+1$, and possibly a finite set of points. But (\ref{I}) is the Hilbert function of the hypersurface of degree $k+1$ alone (in the linear space  $\Lambda$). Thus $[I]_d$ is the degree $d$ component of the saturated ideal of $Y = F$, as claimed.
\end{proof}

This result implies Green's Theorem 3, at least with the stronger assumption on the characteristic in Theorem \ref{T:20160707-213}. (In the correction of Green's Theorem 3 given in the appendix of \cite{BZ}, the assumption on the characteristic is only that $\hbox{char } \k \neq 2$.)

\begin{cor}[Green's Theorem 3] \label{Green 3}
In the previous result, if $k=0$ then $[I]_d$ is the degree $d$ component of the saturated ideal of a linear space of dimension $c$.
\end{cor}

Now we look for conditions that relate the two kinds of extremal behavior. One of these conditions that we will use is reflected in the following lemma.

\begin{lem} \label{WLP cond}
Let $R/I$ be a standard graded algebra and let $L \in [R]_1$ be a general linear form. Let $J = \langle [I]_{\leq d} \rangle$, the ideal generated by the components of $I$ of degree $\leq d$. 

\begin{itemize}

\item[(a)] The following conditions are equivalent.

\begin{itemize}
\item[(i)] $h_d - h_{d+1} + \ell_{d+1} = 0.$

\item[(ii)] The homomorphism 
\[
\times L : [R/I]_{d} \rightarrow [R/I]_{d+1}
\]
is injective.

\item[(iii)]  We have
\[
[J:L]_d = [J]_d =[I]_d.
\]

\end{itemize}

\item[(b)] If the conditions of (a) hold then 
we have an injection
\[
\times L : [R/J]_d \rightarrow [R/J]_{d+1}.
\]

\end{itemize}
\end{lem}

\begin{proof}
Part (a) is immediate from the exact sequence
\[
0 \rightarrow [(I:L)/I]_d \rightarrow [R/I]_d \stackrel{\times L}{\longrightarrow} [R/I]_{d+1} \rightarrow [R/(I,L)]_{d+1} \rightarrow 0.
\]
For part (b), notice that $[I]_d = [J]_d$ and $[J]_{d+1} \subseteq [I]_{d+1}$. Consider the commutative diagram
\[
\begin{array}{ccccccccccc}
&&&&&& 0 \\
&&&&&& \downarrow \\
&& 0 & \rightarrow & [R/J]_d & \rightarrow & [R/I]_d  & \rightarrow & 0 \\
&& \downarrow && \phantom{ \times L} \downarrow { \times L}  && \phantom{ \times L} \downarrow { \times L}  \\
0 & \rightarrow & [I/J]_{d+1} & \rightarrow & [R/J]_{d+1} & \rightarrow & [R/I]_{d+1} & \rightarrow & 0
\end{array}
\]
Then the result follows from the Snake Lemma.
\end{proof}

In the following theorem, we see the effect of two different assumptions on the multiplication by a general linear form on $R/I$. This result is independent of the characteristic.

\begin{thm} \label{relate}
Let $I \subset R = \k[x_0,x_1,\dots,x_n]$ be a homogeneous ideal. Let $L$ be a general linear form. 
Let $J = \langle [I]_{\leq d} \rangle$, the ideal generated by the components of $I$ of degree $\leq d$.  
Assume that for some integer $d$ we have the binomial expansion 
\[
h_d =  \dim_\k [R/I]_d = \binom{a_d}{d} + \binom{a_{d-1}}{d-1} + \dots + \binom{a_e}{e}, \hbox{ where } e  \geq 2.
\]


\begin{itemize}

\item[(a)] Assume that the multiplication $\times L : [R/I]_{d} \rightarrow [R/I]_{d+1}$ is injective.  Then the following conditions are equivalent:

\medskip

\begin{itemize}

\item[(i)]  $\displaystyle \dim [R/(I,L)]_d = (h_d)|^{-1}_{\phantom{-}0}$ (i.e. Green's Theorem 1 is sharp  for $R/I$ in degree $d$);

\smallskip

\item[(ii)] The Hilbert function of $R/(I:L)$ has maximal growth (i.e. Macaulay's theorem is sharp) from degree $d-1$ to degree $d$;

\smallskip

\item[(iii)] The Hilbert function of $R/(J:L)$ has maximal growth (i.e. Macaulay's theorem is sharp) from degree $d-1$ to degree $d$.

\smallskip

\item[(iv)] The Hilbert function of $R/J$ has maximal growth (i.e. Macaulay's theorem is sharp) from degree $d$ to degree $d+1$;

\end{itemize}

\smallskip

\item[(b)] Assume that the multiplication $\times L : [R/I]_{d-1} \rightarrow [R/I]_{d}$ is injective.  Then the following conditions are equivalent:

\begin{itemize}

\item[(i)]  $\displaystyle \dim [R/(I,L)]_d = (h_d)|^{-1}_{\phantom{-}0}$ (i.e. Green's Theorem 1 is sharp  for $R/I$ in degree $d$);

\smallskip

\item[(ii)] The Hilbert function of $R/I$ has maximal growth (i.e. Macaulay's theorem is sharp) from degree $d-1$ to degree $d$;

\smallskip

\item[(iii)] The Hilbert function of $R/(J:L)$ has maximal growth (i.e. Macaulay's theorem is sharp) from degree $d-1$ to degree $d$.

\end{itemize}

\end{itemize}

\end{thm}

\begin{proof}
Notice that $[I]_t = [J]_t$ for all $t \leq d$, but we only have $[J]_{d+1} \subseteq [I]_{d+1}$.  
We first prove (a). 
By the definition of $J$, Green's theorem is sharp for $R/I$ in degree $d$ if and only if it is sharp for $R/J$ in degree $d$. 
Note that by Lemma \ref{WLP cond}, the injectivity assumption for (a) implies the corresponding injectivity for $R/J$ as well.
Thanks to the exact sequences
\[
0 \rightarrow [(J:L)/J]_d \rightarrow [R/J]_d \stackrel{\times L}{\longrightarrow} [R/J]_{d+1} \rightarrow [R/(J,L)]_{d+1} \rightarrow 0
\]
and
\[
0 \rightarrow [(I:L)/I]_d \rightarrow [R/I]_d \stackrel{\times L}{\longrightarrow} [R/I]_{d+1} \rightarrow [R/(I,L)]_{d+1} \rightarrow 0
\]
we obtain
\[
[J:L]_d = [J]_d = [I]_d = [I:L]_d.
\]
Consider the  exact sequence
\begin{equation} \label{exact seq}
0  \rightarrow  [R/(I:L)]_{d-1}  \stackrel{\times L}{\longrightarrow}  [R/I]_{d}  \rightarrow  [R/(I,L)]_{d}  \rightarrow  0 .
\end{equation}
We are given the value of the second vector space in (\ref{exact seq}):

\begin{equation} \label{EQ:20160706-207}
\begin{array}{rcl}
h_d & = & \dim [R/I]_d \\[.5ex] 
& = & \displaystyle \binom{a_d}{d} + \dots + \binom{a_e}{e} .\\ 
\end{array}
\end{equation}
We also know that 
\begin{equation} \label{EQ:20160706-208}
(h_d)|^{-1}_{\phantom{-}0} = \binom{a_d-1}{d} + \dots + \binom{a_e-1}{e}.
\end{equation}
It is worth noting that we are allowing the case $a_e = e$, in which case the last binomial coefficient (and possibly others) in (\ref{EQ:20160706-208}) becomes zero. 
A simple calculation gives
$$
h_d - (h_d)|^{-1}_{\phantom{-}0}= \binom{a_d -1}{d-1} + \dots + \binom{a_e -1}{e-1}.
$$
The exactness of (\ref{exact seq}) then gives that Green's theorem is sharp in degree $d$ if and only if
\[
\dim [R/(I:L)]_{d-1} = \binom{a_d -1}{d-1} + \dots + \binom{a_e -1}{e-1}.
\]
Since $e \geq 2$, this is the $(d-1)$-binomial expansion for $\dim [R/(I:L)]_{d-1}$.
Since $ [I]_d = [I:L]_d$, the Hilbert function of $R/(I:L)$ has maximal growth from degree $d-1$ to degree $d$ if and only if Green's Theorem 1 is sharp for $R/I$ in degree $d$, proving the equivalence of (i) and (ii). The above equalities also immediately give (iii).

For part (a) it remains to prove the equivalence of (iv) to the other three conditions.  Since $J \subset (J:L)$ and $[J:L]_d = [J]_d$, it is clear that (iii) implies (iv). Now we will show that (iv) implies (i). 
Assume that $R/J$ has maximal growth from degree $d$ to degree $d+1$. By the Gotzmann persistence theorem, $R/J$ has maximal growth for all degrees greater than or equal to $d$, and $J$ is $k$-regular for each $k\geq d$. So $J$ is $k$-saturated for each $k\geq d$, and there is a scheme $\X\subset \P^n$ such that 
$$
J_k= [I_\X]_k \quad \text{ for each }k\geq d.
$$
Define
$$
\begin{array}{llllllll}
M(\X)&=&\min\{t \mid \H(R/(I_\X,L),k)=(\H(R/I_\X,k))_{(k)}\mid^{-1}_{\phantom{-}0}  \text{ for each $k\geq t$}\}, & \text{and}\\[.5ex]
G(\X)&=&\min\{t \mid \H(R/I_\X, k+1)=(\H(R/I_\X, k))_{(k)}\mid^{+1}_{+1}  \text{ for each $k\geq t$}\}.
\end{array}
$$
It follows from Proposition~3.1 in \cite{AGS} that $M(\X)\leq G(\X)$ . So, our assumption implies that 
$$ M(\X)\leq G(\X)\leq d,$$
which means 
$$
\begin{array}{lllllllll}
\displaystyle \dim_\k [R/(I,L)]_d 
& = & \dim_\k [R/(J,L)]_d\\
& = & \dim_\k [R/(I_\X,L)]_d\\
& = & (\H(R/I_\X,d))_{(d)}\mid^{-1}_{\phantom{-}0}\quad (\text{since } M(\X)\leq d)\\
& = & (\H(R/J,d))_{(d)}\mid^{-1}_{\phantom{-}0}\\
& = & (h_d)_{(d)}\mid^{-1}_{\phantom{-}0}.
\end{array}
$$
This concludes the proof of (a).

We now assume the injectivity given in (b). Then we have 
\[
[J:L]_{d-1} = [J]_{d-1} =[I]_{d-1} = [I:L]_{d-1} \hbox{ and } [J]_d = [I]_d.
\]
These equalities and the same calculation as in (a) give that (i) is equivalent to (ii). 
To see that  (ii) implies (iii), suppose that the Hilbert function of $R/J$ (equivalently $R/I$) has maximal growth from degree $d-1$ to degree $d$. By Gotzmann's theorem, the ideal  
$J$ is $(d-1)$-regular (and hence $(d-1)$-saturated). This implies that 
\[
(J:L)_k=J_k\quad \text{ for all }k\geq d-1.
\]
Hence, the Hilbert function of $R/(J:L)$ has maximal growth from degree $d-1$ to degree $d$.

Finally we prove that (iii) implies (i). For convenience, we use a small variation on the notation in Notation \ref{notation1}: 

\begin{itemize}
  \item $h_d=\dim [R/J]_d = \dim [R/I]_d;$
  \item $b_d=\dim [R/(J:L)]_d;$
  \item $\ell_d=\dim [R/(J,L)]_d = \dim [R/(I,L)]_d;$
 \end{itemize}
  
\noindent Suppose that 
$$(b_{d-1})_{(d-1)}\mid^{+1}_{+1}=b_d.$$
This implies that $(J:L)$ has no generators of degree $d$. So, we have
$$ J_d\subset (J:L)_d={\mathfrak m} (J:L)_{d-1} = {\mathfrak m} (J)_{d-1}\subset J_d,$$
where ${\mathfrak m}$ is the maximal ideal of $R$. This means that 
$$J_d= (J:L)_d,$$
and thus 
$$b_{d-1}=h_{d-1}\quad \hbox{and} \quad b_d=h_d.$$
By the assumption that 
$$(b_{d-1})_{(d-1)}\mid^{+1}_{+1}=b_d,$$
we have that 
$$
\begin{array}{lllllll}
(b_{d-1})_{(d-1)}\mid^{+1}_{+1}& = &(h_{d-1})_{(d-1)}\mid^{+1}_{+1} \\
        &=& [(h_d-\ell_d)]_{(d-1)}|^{+1}_{+1}\\
        &\geq &\left[h_{d} - ((h_{d})_{(d)}) \mid^{-1}_{\phantom{-}0} \right]_{(d-1)}\mid^{+1}_{+1}\\
				        &=& h_{d} \quad (\hbox{since } e\geq 2)\\
				        &=& b_{d} \\
				        &=& (b_{d-1})_{(d-1)}\mid^{+1}_{+1}.
\end{array}
$$\
Since the function $(-)\mid^{+1}_{+1}$ is strictly increasing,
we see that $ h_d-\ell_d =h_{d} - (h_{d})_{(d)} \mid^{-1}_{\phantom{-}0}$, and thus
$$
\ell_d=(h_{d})_{(d)} \mid^{-1}_{\phantom{-}0},
$$
as we wished.
\end{proof}

\begin{exmp}
Let $C$ be a smooth rational quartic curve in $\mathbb P^3$. Note that $\hbox{depth } R/I_C = 1$, so $\times L$ is injective in all degrees, for a general linear form $L$. We have the following decomposition for the Hilbert function:

\medskip

\begin{center}
\begin{tabular}{c|cccccccccccccccc}
deg & 0 & 1 & 2 & 3 & 4 & 5 & 6 & 7 & 8 & \dots \\ \hline
& 1 & 4 & 9 & 13 & 17 & 21 & 25 & 29 & 33 & \dots \\
& & 1 & 4 & 9 & 13 & 17 & 21 & 25 & 29 & \dots  \\ \hline
& 1 & 3 & 5 & 4 & 4 & 4 & 4 & 4 & 4 & \dots
\end{tabular}
\end{center}

\noindent and we have
\[
\begin{array}{rcl}
21 & = & \displaystyle \binom{7}{5} \\ \\
25 & = & \displaystyle \binom{7}{6} + \binom{6}{5} + \binom{5}{4} + \binom{4}{3} + \binom{3}{2} \\ \\
29 & = & \displaystyle \binom{8}{7} + \binom{7}{6} + \binom{6}{5} + \binom{5}{4} + \binom{3}{3} + \binom{2}{2} + \binom{1}{1} \\ \\
33 & = & \displaystyle \binom{9}{8} + \binom{8}{7} + \binom{7}{6} + \binom{6}{5} + \binom{4}{4} + \binom{3}{3} + \binom{2}{2}
\end{array}
\]
Note that Macaulay's theorem is sharp from degree 7 to degree 8 and from then on, Green's theorem is sharp from degree 7 on, and $e \geq 2$ from degree 8 on. This shows that without the condition $e \geq 2$ the theorem is false, since sharpness of Green's theorem in degree $d=7$ does not imply maximal growth for $R/(I_C : L)$ from degree $d-1=6$ to degree $d=7$.
\end{exmp}

\medskip

\begin{cor} \label{cor of thm}
Assume either the equivalent conditions in (a) or the equivalent conditions of (b) in Theorem~\ref{relate}. 

\begin{itemize}

\item[(i)] The Hilbert function of $R/(J:L)$ has maximal growth in all degrees $\geq d$ (i.e. Macaulay's theorem is sharp for $R/J$).

\item[(ii)] The component $[I]_d$ defines a closed subscheme $\X \subset \mathbb P^n$, and we have for all $t \geq d-1$, $[J]_t = [I_\X]_t$. 

\item[(iii)] The Hilbert polynomial $P_\X$ of $\X$ is characterized by 
\[
P_\X(d+t) = \binom{a_d+t}{d+t} + \dots + \binom{a_e +t}{e+t}.
\]

\item[(iv)] Suppose that $a_e>e$. Then, there is a $(a_d-d+1)$-dimensional linear space $\Lambda\subset \P^n$ such that $$
\X\subset \Lambda.
$$ 
Moreover, the Hilbert function of $R/I_\X$ is entirely determined by recursive process with the equation 
$$
\H_\X(k-1)=\H_\X(k)-\H_\X(k)|^{-1}_{\phantom{-}0} \quad \text{for all } k\leq d.
$$
\end{itemize}

\end{cor}

\begin{proof}

We apply Gotzmann's theorem. Assuming either (a) or (b) of Theorem \ref{relate}, we have that the Hilbert function of $R/(J:L)$ has maximal growth from degree $d-1$ to degree $d$. Then since $J$ has no new generators in higher degrees, by Gotzmann's theorem, the same is true in all higher degrees. This is (i). In particular, both $[J:L]_{d-1}$ and $[J:L]_d$ define the same scheme $\mathbb X \subset \mathbb P^{n}$.

In both parts of Theorem \ref{relate} we showed that $ (J:L)_d = J_d $ (which is also equal to $[I]_d$ by definition of $J$). 
Then Gotzmann's theorem provides (ii) and (iii).

We now prove (iv). Let $h'_k:=\dim_\k[R/I_\X]_k$ and $\ell '_k:=\dim_\k[R/(I_\X,L)]_k$. Since $I_\X$ is saturated, we see that the multiplication map by a general linear form $L$
$$
\times L : [R/I_\X]_{k} \rightarrow [R/I_\X]_{k+1}
$$
 is injective for all $k\geq 0$. This means that 
 $$ \Delta h'_k= h'_k-h'_{k-1}=\ell'_k\quad\text{ for all }k\geq 0.$$

Consider the $d$-th binomial expansion of $h'_d$
$$
\displaystyle  h'_d = \binom{a_d}{d} + \binom{a_{d-1}}{d-1} + \dots + \binom{a_{e}}{e}.
$$ 
By assumption we have 
$$
\quad \ell'_d = (h'_d) \mid^{-1}_{\phantom{-}0}.
$$
For general linear forms $L$ and $L'$,
\begin{equation} \label{string}
 \begin{array}{llllllllll}
  (h'_{d-1})_{(d-1)}\mid ^{-1}_{\phantom{-}0} &\geq& \ell'_{d-1}\\[1ex]
          & \geq &  \ell'_{d}- \dim_\k[R/(I_\X,L,L')]_{d}\\[1ex]
          &=& (h'_d)_{(d)}\mid ^{-1}_{\phantom{-}0} - \dim_\k[R/(I_\X,L,L')]_{d}\\[1ex]
          &\geq& (h'_d)_{(d)}\mid ^{-1}_{\phantom{-}0} - (h'_d)_{(d)}\mid^{-2}_{\phantom{-}0}.
 \end{array}
\end{equation}

\noindent Now we will show that the first and last of these are equal, making all the intermediate values equal as well.

\medskip

\noindent \underline{Claim}: if $a_{e}> e$ then $\displaystyle  (h'_{d-1})_{(d-1)}\mid ^{-1}_{\phantom{-}0}= (h'_d)_{(d)}\mid ^{-1}_{\phantom{-}0} - (h'_d)_{(d)}\mid^{-2}_{\phantom{-}0}.$

\medskip

The claim follows by the same argument as in the proof of Lemma 3.1 in \cite{AGS}, but we include the details for completeness.

\medskip

By the assumption that  $\ell'_d=(h'_d)_{(d)}\mid^{-1}_{\phantom{-}0},$
we have 
{\small 
\begin{equation}\label{EQ:20160908-001}
\begin{array}{llllllll}
h'_{d-1} & = & h'_d-\ell'_d\\[.5ex]
& = & h'_d-[(h'_d)_{(d)}]\mid^{-1}_{\phantom{-}0}\\[1ex]
& = & \ds \left[\binom{a_d}{d} + \binom{a_{d-1}}{d-1} + \dots + \binom{a_e}{e}\right] -\left[\binom{a_d-1}{d} + \binom{a_{d-1}-1}{d-1} + \dots + \binom{a_e-1}{e}\right]\\[2ex]
            &=& \left\{
\begin{array}{lll}
\ds\binom{a_d-1}{d-1} + \binom{a_{d-1}-1}{d-2}+\dots + \binom{a_{e}-1}{e-1}, & \text { if }e\geq 2,\\[2ex]
\ds\binom{a_d-1}{d-1} + \binom{a_{d-1}-1}{d-2}+\cdots + \binom{a_{\delta+1}-1}{\delta} +\binom{a_{\delta}}{\delta-1}, & \text { if } e=1,
\end{array}
\right.
\end{array}
\end{equation} }
where $\delta=\max\{\,i\geq 2\mid a_i-i=a_2-2 \}$ and the latter is a routine calculation.

Hence we obtain that   
$$
(h'_{d-1})_{(d-1)}\mid^{-1}_{\phantom{-}0}=\left\{
\begin{array}{lll}
\ds\binom{a_d-2}{d-1} + \binom{a_{d-1}-2}{d-2}+\dots + \binom{a_{e}-2}{e-1}, & \text { if }e\geq 2,\\[2ex]
\ds\binom{a_d-2}{d-1} + \binom{a_{d-1}-2}{d-2}+\cdots + \binom{a_{\delta+1}-2}{\delta} +\binom{a_{\delta}-1}{\delta-1}, & \text { if } e=1.
\end{array}
\right.
$$
On the other hand, since $a_e > e$ we have 
$$
\begin{array}{llllllll}
   & (h'_d)_{(d)}\mid ^{-1}_{\phantom{-}0} - (h'_d)_{(d)}\mid^{-2}_{\phantom{-}0}\\[.5ex] 
= & \ds \left[\binom{a_d-1}{d} + \binom{a_{d-1}-1}{d-1} + \dots + \binom{a_e-1}{e}\right] -\left[\binom{a_d-2}{d} + \binom{a_{d-1}-2}{d-1} + \dots + \binom{a_e-2}{e}\right]\\[2ex]
= & \left\{
\begin{array}{lll}
\ds\binom{a_d-2}{d-1} + \binom{a_{d-1}-2}{d-2}+\dots + \binom{a_{e}-2}{e-1}, & \text { if }e\geq 2,\\[2ex]
\ds\binom{a_d-2}{d-1} + \binom{a_{d-1}-2}{d-2}+\cdots + \binom{a_{\delta+1}-2}{\delta} +\binom{a_{\delta}-1}{\delta-1}, & \text { if } e=1
\end{array}
\right.
\end{array}
$$
and the claim is proved.

So we have equalities in (\ref{string}), and hence
$$
\ell'_{d-1}=(h'_{d-1})_{(d-1)}\mid^{-1}_{\phantom{-}0}.
$$
Let $\binom{b_e'}{e'}$ be the last binomial coefficient in the $(d-1)$-st binomial expansion of $h'_{d-1}$. Then, by equation~\eqref{EQ:20160908-001}, we see that 
$$
(b_{e'},e')=\left\{
\begin{array}{llll}
(a_e-1,e-1)& \text { if }e\geq 2,\\[1ex]
(a_{\delta},\delta-1) & \text { if } e=1,
\end{array}
\right.
$$
and hence 
$$
b_{e'}> e' \quad \text{ and } \quad \ell'_{d-1}=(h'_{d-1})_{(d-1)} \mid^{-1}_{\phantom{-}0}.
$$ 
Replace $(h'_d,\ell'_{d})$ by $(h'_{d-1},\ell'_{d-1})$ and repeat the argument up to the degree $d=1$. This implies that 
$$
h'_{k-1}=h'_k-(h'_k)_{(k)}\mid ^{-1}_{\phantom{-}0}, \quad \text{ for each $k\leq d$.}
$$
Moreover, one can show that
$$
h'_1= a_d-d+2,
$$
which implies $\X$ is contained in a $(a_d-d+1)$-dimensional linear subspace $\Lambda\subset \mathbb P^n$. 
\end{proof}

\noindent
\begin{rem} By Theorem~4.7 in \cite{AGS}, if $\X$ is a reduced equidimensional closed subscheme in $\P^n$, which is not a hypersurface in a linear
subspace in $\P^n$, then $a_e=e$ so there is no contradiction with part (iv) of Corollary~\ref{cor of thm}.
\end{rem}

\begin{cor}
Let $R/I$ be an Artinian algebra with the weak Lefschetz property (e.g. a height 3 complete intersection). Assume that in degree $d$ we have $\dim [R/I]_d \leq \dim [R/I]_{d+1}$. Assume that the binomial expansion of $h_d$ satisfies the numerical assumption in Theorem \ref{relate}. Assume also that the linear system defined by $[I]_d$ is basepoint free. Then Green's theorem is not sharp for $R/I$ in degree $d$.
\end{cor}

\begin{exmp}
In $\k[x_0,x_1,x_2]$ let $I$ be the complete intersection of three forms of degree $6$. We have
$$
\dim [R/I]_6 = 25 = \binom{7}{6} + \binom{6}{5} + \binom{5}{4} + \binom{4}{3} + \binom{3}{2}
$$
and
$$
\dim[R/(I,L)]_6 = 4 <  \binom{6}{6} + \binom{5}{5} + \binom{4}{4} + \binom{3}{3} + \binom{2}{2},
$$
i.e. Green's theorem is not sharp there.
\end{exmp}

We now give a small variation on Theorem \ref{relate}, showing how it is improved by a slightly stronger assumption.

\begin{cor}\label{C:20160706-215} 
Let $I \subset R = \k[x_0,\dots,x_n]$ be a homogeneous ideal, where $\k$ is algebraically closed. Let $L_1,\dots,L_s$ be general linear forms. 

Assume that for some integer $d$ we have

\medskip

\begin{itemize}

\item[(a)] $\displaystyle  h_d =  \dim [R/I]_d = \binom{a_d}{d} + \binom{a_{d-1}}{d-1} + \dots + \binom{a_e}{e}$ where $e  \geq 2$;  

\medskip

\item[(b)] Green's Theorem 1 is sharp for $R/I$ in degree $d$.

\medskip

\end{itemize}

\noindent Then Green's Theorem 1 is successively sharp restricting modulo $L_1,\dots,L_s$.

\end{cor}

\begin{proof}
We use the calculations from the previous result.
Consider the  diagram 

\begin{equation} \label{EQ:20160706-209}
\begin{array}{cccccccccccccc}
&& 0 && 0 && 0 \\
&& \downarrow && \downarrow && \downarrow \\
 && [R/((I:L_1):L_2) ]_{d-2} && [R/(I:L_2)]_{d-1} && [R/((I,L_1):L_2)]_{d-1} \\
&& \phantom{\scriptstyle \times L_2} \downarrow \scriptstyle \times L_2 &&  \phantom{\scriptstyle \times L_2} \downarrow \scriptstyle \times L_2  &&  \phantom{\scriptstyle \times L_2} \downarrow \scriptstyle \times L_2  \\
0 & \rightarrow & [R/(I:L_1)]_{d-1} & \stackrel{\times L_1}{\longrightarrow} & [R/I]_{d} & \rightarrow & [R/(I,L_1)]_{d} & \rightarrow & 0 \\
&& \downarrow && \downarrow && \downarrow \\
&& [R/((I:L_1),L_2)]_{d-1} && [R/(I,L_2)]_{d} && [R/(I,L_1,L_2)]_{d} \\
&& \downarrow && \downarrow && \downarrow \\
&& 0 && 0 && 0 
\end{array}
\end{equation}

\noindent Looking at the first column of (\ref{EQ:20160706-209}), Green's Theorem 1 then gives
$$
\dim [R/((I:L_1),L_2)]_{d-1} \leq \binom{a_d-2}{d-1} + \dots + \binom{a_e-2}{e-1}.
$$
It is important to note that this holds even if $a_e = e$. What is important is the condition $e \geq 2$.

Because we assumed that Green's Theorem 1 is sharp for $R/I$ in degree $d$, we can apply Green's Theorem~1 again to $R/(I,L_1)$ and we have 
$$
\dim [R/(I,L_1,L_2)]_{d} \leq (h_{d})|^{-2}_{\phantom{-}0} =  \binom{a_d-2}{d} + \dots + \binom{a_e-2}{e}.
$$
Since
$$
((I,L_2):L_1) \supseteq ((I:L_1),L_2),
$$
we obtain
$$
\begin{array}{rcl}
\dim [R/((I,L_2):L_1)]_{d-1} & \leq & \displaystyle \dim [R/((I:L_1),L_2)]_{d-1} \\ \\
& \leq & \displaystyle \binom{a_d-2}{d-1} + \dots + \binom{a_e-2}{e-1}.
\end{array}
$$
The sequence
$$
0 \rightarrow [R/((I,L_2):L_1)]_{d-1} \rightarrow [R/(I,L_2)]_{d} \rightarrow [R/(I,L_1,L_2)]_{d} \rightarrow 0
$$
then gives 
$$
\begin{array}{rcl}
\displaystyle  \binom{a_d-1}{d} + \dots + \binom{a_e-1}{e}  & = & \dim [R/(I,L_2)]_d \\ 
& = & \dim [R/((I,L_2):L_1)]_{d-1} + \dim [R/(I,L_1,L_2)]_{d} \\[.5ex]  
& \leq & \displaystyle \left [   
\binom{a_d-2}{d-1} + \dots + \binom{a_e-2}{e-1} \right ] +
\left [
\binom{a_d-2}{d} + \dots + \binom{a_e-2}{e}
\right ] \\[2ex] 
& = & \displaystyle \binom{a_d-1}{d} + \dots + \binom{a_e-1}{e}.
\end{array}
$$
Again notice that this holds even if $a_e = e$. We conclude
$$
\dim [R/((I,L_2):L_1)]_{d-1} =   \binom{a_d-2}{d-1} + \dots + \binom{a_e-2}{e-1}
$$
and
\begin{equation} \label{{EQ:20160706-210}}
\begin{array}{rcl}
\dim [R/(I,L_1,L_2)]_{d} &  = &  \displaystyle  \binom{a_d-2}{d} + \dots + \binom{a_e-2}{e} \\ \\
& = & (h_d)|^{-2}_{\phantom{-}0}.
\end{array}
\end{equation}
Now replace $I$ by $(I,L_1)$ and repeat the argument, adding one linear form at a time, continuing through~$L_s$.
\end{proof}

The following result will be useful in the next section. It makes only an injectivity assumption (equivalent to a certain numerical assumption, as noted in Lemma \ref{WLP cond}). Note also the similarity to Lemma \ref{L:20160525-205}, and to \cite[Proposition 2.1 (b)]{MMN}, although this proof is completely different. In the next section this will be of use to us.

\begin{lem} \label{socle}
For a general linear form $L$ in $R = \k[x_0,\dots,x_n]$ assume that 
\[
\times L : [R/I]_{d} \rightarrow [R/I]_{d+1}
\]
is injective (i.e. $h_d - h_{d+1} + \ell_{d+1} = 0$), and that 
\[
\times L : [R/I]_{d-1} \rightarrow [R/I]_{d}
\]
has an $s$-dimensional kernel  (i.e. $h_{d-1} - h_{d} + \ell_{d} = s$).
Then $R/I$ has an $s$-dimensional socle in degree~$d-1$.
\end{lem}

\begin{proof}
Let $L_0,\dots,L_n$ be $n+1$ general linear forms. Note that they form a basis for $[R]_1$. Choose any two, $L_i$ and $L_j$ and consider the following commutative diagram:
\[
\begin{array}{cccccccccccccccc}
&&  && 0 \\
&&  && \downarrow \\
&&  && \phantom{d-} \left [ \frac{I:L_j}{I} \right ]_{d-1} && 0 \\
&&  && \downarrow && \downarrow \\
0 & \rightarrow &  \phantom{d-} \left [ \frac{I:L_i}{I} \right ]_{d-1} & \rightarrow & \phantom{d-} [R/I]_{d-1} & \stackrel{\times L_i}{\longrightarrow} & [R/I]_d & \rightarrow & [R/(I,L_i)]_d & \rightarrow & 0 \\
&& && \phantom{\scriptstyle \times L_j} \downarrow {\scriptstyle \times L_j} &&
 \phantom{\scriptstyle \times L_j} \downarrow {\scriptstyle \times L_j}  \\
&& 0 & \rightarrow & [R/I]_{d} & \stackrel{\times L_i} {\longrightarrow} & \phantom{d-}  [R/I]_{d+1} & \rightarrow & [R/(I,L_i)]_{d+1} & \rightarrow & 0 \\
\end{array}
\]
An easy diagram chase shows that the kernel of multiplication by $L_i$ is the same as the kernel of multiplication by $L_j$. Since $L_0,\dots,L_n$ form a basis for $[R]_1$, this kernel is contained in the kernel of multiplication by any linear form, and we are done.
\end{proof}


\section{Classification of Gorenstein sequences of the form $(1,r,r-2,r,1)$}  \label{GorSection}

As indicated in the introduction, a great deal of research has gone into the study of possible Gorenstein Hilbert functions (i.e. Gorenstein sequences). As a subproblem, it has been of great interest to understand when they can be unimodal. In the recent paper \cite{MZ:2} this  was solved for socle degrees 4 and 5. However, this fell short of a classification of the possible Hilbert functions even in socle degree 4 -- what is missing is to completely understand the extent of non-unimodality that occurs. What is now known thanks to that paper is a classification of the possible Gorenstein Hilbert functions of the form $(1,r, r-1,r,1)$. However, even the case $(1,r,r-2,r,1)$ is open. In this section we complete this case, as well as an analogous one for socle degree 5 (see Corollary~\ref{r-2}). 

In \cite{{MZ:2}} it was observed in Remark 3.5 that $(1,20,18,20,1)$ is a Gorenstein Hilbert function, arising easily using trivial extensions. It will then follow from Theorem \ref{wellknown} that this is the smallest possible of the form $(1,r,r-2,r,1)$ (and hence all values $r \geq 20$ also exist) once we show that the $h$-vector $\H=(1,19,17,19,1)$ is not a Gorenstein sequence. We do this using results from the previous section. We will use without comment Notation \ref{decomp notation}. The characteristic assumption is only to be able to use Theorem~\ref{T:20160707-213}.

\begin{thm} \label{T:20160707-301}
Assume that $\hbox{char } \k = 0$. Then the $h$-vector $\H=(1, 19,17,19,1)$ is not Gorenstein.
\end{thm} 

\begin{proof} 
Assume that there exists an Artinian Gorenstein algebra $R/I$ with Hilbert function $\H$. 
Let $J= \langle I_{\leq 3} \rangle$ be the ideal generated by the components of $I$ in degrees $\leq 3$. Then, by Macaulay's theorem, 
$$
\H(R/J,4)\leq 31. 
$$

\begin{itemize}

\item[(i)] If $\H(R/J,4)= 31$ then the Hilbert function of $R/J$ has maximal growth in degree $3$. So by Lemma~\ref{L:20160525-205}, $R/J$ has a $7$-dimensional socle elements in degree $2$, and hence so does $R/I$. This contradicts the Gorenstein assumption.

\medskip

\item[(ii)] If $\H(R/J,4)= 30$, then the Betti table of $R/J^{\rm lex}$ (after truncating in degree 4) is of the form 
$$
\begin{array}{c|ccccccccccccccccccccccccccccc}
                                     & 0 & 1 & \cdots & \cdots & 18 & 19 & \\ 
                             \hline                                
                                0 & 1& 0 &  \cdots & \cdots & 0&  0 &\\ 
                                1 & 0 &173  &  \cdots & \cdots &247 & 13 & &\\
                                2 & 0 & 19 &  \cdots & \cdots & 131&  \text{\bfseries 7} &\\ 
                                3 & 0 &1  &  \cdots & \cdots &\text{\bf 0} &  0 & &\\
                                4 & 0 &43  &  \cdots & \cdots &551 & 30 & &
\end{array}
$$ 
Using the Cancellation Principle (see \cite{P:1}), we get that $R/J$ has a socle element in degree $2$, and hence so does $R/I$, which is a contradiction.  

\medskip

\item[(iii)] If $\H(R/J,4)= 29$, then the Betti table of $R/J^{\rm lex}$ (after truncating in degree 4) is of the form 
$$
\begin{array}{c|ccccccccccccccccccccccccccccc}
                                     & 0 & 1 & \cdots & \cdots & 18 & 19 & \\ 
                             \hline                                
                                0 & 1& 0 &  \cdots & \cdots & 0&  0 &\\ 
                                1 & 0 &173  &  \cdots & \cdots &247 & 13 & &\\
                                2 & 0 & 19 &  \cdots & \cdots & 131&  \text{\bfseries 7} &\\ 
                                3 & 0 &2  &  \cdots & \cdots &\text{\bf 1} &  0 & &\\
                                4 & 0 &41  &  \cdots & \cdots &532 & 29 & &
\end{array}
$$ 
Again using the Canclellation Principle, we get that $R/J$ has a socle element in degree $2$, and hence so does $R/I$, which is a contradiction.  
\end{itemize}

\smallskip

As a result, we have:

\begin{center}

{\em Without loss of generality we can assume that $\H(R/J,4)\le 28$. }

\end{center}

Since we must have
\[
\ell_3 \leq (\ell_2)_{(2)}|^{+1}_{+1} \hbox{ and } \ell_3 \leq (h_3)_{(3)}|^{-1}_{\phantom{-}0}
\]
as well has having the middle line symmetric (Gorenstein), 
there are three possibilities for the decomposition of $\H$. They are
$$
\begin{array}{ccccccc}
\begin{array}{ccccccccccccc}
1 & 19 & 17 & 19 & 1 \\
   & 1   & 10 & 10 & 1 \\ \hline 
1 & 18 &   7 &   9 &    
\end{array}
&&&
\begin{array}{ccccccccccccc}
1 & 19 & 17 & 19 & 1 \\
   & 1   & 11 & 11 & 1 \\ \hline 
1 & 18 &  6 &   8 &    
\end{array}
&&&
\begin{array}{ccccccccccccc}
1 & 19 & 17 & 19 & 1 \\
   & 1   & 12 & 12 & 1 \\ \hline 
1 & 18 &   5 &   7 &    
\end{array}
\end{array} 
$$

\medskip

\ni
{\be Case 1.} We consider the first decomposition of $\H$, namely 
\begin{equation} \label{EQ:20160601-308}
\begin{array}{ccccccccccccc}
1 & 19 & 17 & 19 & 1 \\
   & 1   & 10 & 10 & 1 \\ \hline 
1 & 18 &   7 &   9 &    
%
\end{array} 
\end{equation}
%
%
Since 
$$
h_{(3)}=19_{(3)}=\binom{2+3}{3}+\binom{2+2}{2}+\binom{2+1}{1} \quad \text{and} \quad 
\ell_3=9=19_{(3)}\big|^{-1}_{\phantom{-}0},
$$
by Theorem~\ref{T:20160707-213}, there is a $3$-dimensional linear space $\Lambda$ such that $J_3$ defines a hypersurface $F$ of degree $3$ in $\Lambda\subset \P^{18}$ (and is saturated in degree 3). Since $J$ is generated in degrees $\le 3$, the Hilbert function of $R/J$ is
$$ 
\H(R/J, t)=\binom{2+t}{t}+\binom{2+(t-1)}{(t-1)}+\binom{2+(t-2)}{(t-2)}, \quad \text{ for all } t\geq 3,
$$
and so it has maximal growth in degree $3$. Then, we have 
$$
31=\H(R/J,4)\le 28,
$$
which is a contradiction.

\medskip

\noindent {\be Case 2.}

Assume that we have the decomposition
\[
\begin{array}{ccccccccccccc}
1 & 19 & 17 & 19 & 1 \\
   & 1   & 11 & 11 & 1 \\ \hline 
1 & 18 &   6 &   8 &    
\end{array} 
\]
We have seen that ${\bf H}(R/J,4) \leq 28$. We will consider one further restriction. 
Let $L_1$ and $L_2$ be general linear forms. 
By Green's theorem (see Theorem~\ref{T:201}),
$$
\H(R/(J,L_1,L_2),2)\le ((\ell_2)_{(2)})\big|^{-1}_{\phantom{-}0} =(6_{(2)})\big|^{-1}_{\phantom{-}0} =3.
$$
Consider the  exact sequence
\begin{equation}\label{EQ:20160603-303}
0 \to ((J, L_1) : L_2)/(J, L_1)(-1) \to [R/((J, L_1) ] (-1) \stackrel{\times L_2}{\longrightarrow} R/(J, L_1) \to R/(J,L_1,L_2)\to 0.
\end{equation}
Then we have 
\begin{equation} \label{EQ:20160602-310} 
2=8-6=\ell_3-\ell_2\le \H(R/(J,L_1,L_2),3) \le ((\ell_3)_{(3)})\big|^{-1}_{\phantom{-}0} = (8_{(3)})\big|^{-1}_{\phantom{-}0}=2.
\end{equation} 
So
\begin{equation}\label{EQ:20160801-305}
\H(R/(J,L_1,L_2),3)=2\quad \text{ and }\quad ((J,L_1):L_2)_2=(J,L_1)_2,
\end{equation}
and so by Macaulay's theorem,
$$
\H(R/(J,L_1,L_2),2)=2 \text{ or } 3.
$$
We consider these two cases separately.

\medskip

\begin{itemize}

\item[(a)] Assume $\H(R/(J,L_1,L_2),2)=2$. 
We have the following decomposition for $R/(J,L_1)$:

\begin{equation} 
\begin{array}{ccccccccccccc}
1 & 18 & 6 & 8 &  \\
   & 1   & 4 & 6 &  \\ \hline 
1 & 17 &   2 &   2 &    
\end{array} 
\end{equation}
Since $\H(R/(J,L_1,L_2,3)=2$, by Gotzmann's persistence theorem, $(J,L_1,L_2)$ is $2$-regular. In particular, $[(J,L_1,L_2)]_2$ is the saturated ideal of a zero-dimensional scheme of degree 2, and the same is true if we replace $L_2$ by another general linear form, $L$.   From the commutative diagram
\[
\begin{array}{ccccccccccccc}
&&0  &&&& 0 \\
&& \downarrow &&&& \downarrow \\
0 & \rightarrow & [R/(J,L_1)]_2 & \stackrel{\times L}{\longrightarrow}  & [R/(J,L_1)]_3 & \rightarrow & [R/(J,L_1,L)]_3 & \rightarrow & 0 \\
&& \phantom{\scriptstyle \times L_2} \downarrow {\scriptstyle \times L_2} && \phantom{\scriptstyle \times L_2} \downarrow {\scriptstyle \times L_2}  && \phantom{\scriptstyle \times L_2} \downarrow {\scriptstyle \times L_2}  \\
0 & \rightarrow & [R/(J,L_1)]_3 & \stackrel{\times L}{\longrightarrow} & [R/(J,L_1)]_4 & \rightarrow & [R/(J,L_1,L)]_4 & \rightarrow & 0 \\
\end{array}
\]
we see that 
\[
\times L_2 : [R/(J,L_1)]_3 \rightarrow [R/(J,L_1)]_4
\]
is also injective. Hence we have the decomposition 
\begin{equation} 
\begin{array}{ccccccccccccc}
1 & 18 & 6 & 8 & 10 \\
   & 1   & 4 & 6 & 8  \\ \hline 
1 & 17 &   2 &   2 &  2
\end{array} 
\end{equation}
Then, the decomposition of the Hilbert function of $R/J$ is of the form
\begin{equation} \label{EQ:20160602-311}
\begin{array}{ccccccccccccc}
1 & 19 & 17 & 19 & 28-\alpha \\
   & 1   & 11 & 11 & 18+\alpha \\ \hline 
1 & 18 &   6 &   8 & 10   
\end{array} 
\end{equation}
Moreover, since
$$
28_{(4)}|^{-1}_{\phantom{-}0}=10 \quad \text{and} \quad 27_{(4)}|^{-1}_{\phantom{-}0} =9,
$$
$\alpha$ has to be $0$. 

Notice that the growth of the Hilbert function of $R/(J,L_1)$ from degree 3 to degree 4 is maximal, so by Gotzmann's Persistence Theorem the ideal is saturated in all degrees $\geq 3$ and the Hilbert polynomial is $2t+2$. 
In particular, in all degrees $\geq 3$, $(J,L_1)$ defines either the union of a plane curve of degree 2 and a point (embedded or not) or two skew lines in $\mathbb P^3$. This means that $J^{\rm sat}$ defines the union of  a scheme $\X$ and (possibly embedded) a finite set of $m$ points (for some $m \geq 0$), where $\X$ is either the non-degenerate union of two planes in $\mathbb P^4$ or the union in $\mathbb P^3$ of a quadric surface and a line. Notice that in the first case ${\bf H}(R/I_\X,4) = 29+m$ and in the second case ${\bf H}(R/I_X,4) = 28+m$.

Since
\[
{\bf H}(R/I_\X,4) \leq {\bf H}(R/J^{\rm sat},4) \leq {\bf H}(R/J,4) = 28,
\]
we see that $\X$ is not the union of two planes in $\mathbb P^4$, and furthermore we have $m=0$. Since $28 = \binom{6}{4} + \binom{5}{3} + \binom{3}{2}$, $J$ is already saturated in degree 4 and the Hilbert function has maximal growth from this point on. Then by Lemma \ref{L:20160525-205}, $R/J$ has a $1$-dimensional socle  in degree $3$, hence so does $R/I$, which is a contradiction. 

\medskip

\item[(b)] Assume that $\H(R/(J,L_1,L_2),2)=3$. Then
$$
\begin{array}{rlclll}
\H((R/(J,L_1),2)&=&\ell_2&=&\binom{4}{2}, \text{ and }\\[1ex]
\H((R/(J,L_1,L_2),2)&=&(\ell_2)_{(2)}|^{-1}_{\phantom{-}0} &=&\binom{3}{2}.
\end{array}
$$
By Corollary~\ref{Green 3}, there is a two-dimensional linear space $\Lambda\subset \P^{18}$ such that 
$$(J,L_1)_2=(I_{\Lambda})_2.$$  
We have the decomposition of the Hilbert function of $R/(J,L_1)$ as follows:

\[
\begin{array}{ccccccccccccc}
1 & 18 & 6 & 8 &   \\
   & 1   & 3 & 6 &   \\ \hline 
1 & 17 &   3 &  2   
\end{array} 
\]
Since $\dim_\k [R/I_\Lambda]_3 = 10$, there are two cubic polynomials, $F_1$ and $F_2$, such that $(J,L_1)_3 = [I_\Lambda]_3 + \langle F_1,F_2 \rangle_3$. Letting $\bar F_1$ and $\bar F_2$ be the restrictions to $R/I_\Lambda$, we have the following possibilities (recalling that $J$ is generated in degree $\leq 3$):

\medskip

\begin{itemize}
 \item[(i)] If $\bar F_1$ and $\bar F_2$ are a complete intersection, then 
 the Hilbert function of $R/(J,L_1)$ is
$$
\begin{array}{lllllllllllll}
\H_{R/(J,L_1)}: & 1 & 3 & 6 & 8 & 9 & 9 & \cdots.
\end{array}
$$

 \item[(ii)] If $\bar F_1$ and $\bar F_2$ have a linear common factor, then 
$$
\begin{array}{lllllllllllll}
\H_{R/(J,L_1)}: & 1 & 3 & 6 & 8 & 9 & 10 & \cdots.
\end{array}
$$
 \item[(iii)] If $\bar F_1$ and $\bar F_2$ have a quadratic common factor, then 
$$
\begin{array}{lllllllllllll}
\H_{R/(J,L_1)}: & 1 & 3 & 6 & 8 & 10 & 12 & \cdots.
\end{array}
$$
\end{itemize}

\medskip

\ni In all of these cases we have that for any $d \geq 3$,
$$
[(J,L_1)]^{\sat}_d =(J,L_1)_d \subseteq ((J:L_2),L_1)_d\subseteq (J^{\sat},L_1)_d\subseteq [(J,L_1)]^{\sat}_d.
$$
Hence, 
$$
\begin{array}{llllll}
\H(R/((J:L_2),L_1),3)&=&\H(R/(J,L_1),3)=8 . \\
\end{array}
$$
Remembering that both $L_1$ and $L_2$ are general linear forms and that
\[
{\bf H}(R/(J:L_1),3) \leq h_3 = 19,
\] 
this means that ${\bf H}(R/(J:L_1),3)$ is either 19 or 18.

Now recall that we assume that  $\H(R/J,4)\leq 28$ and
$$
(27)_{(4)}|^{-1}_{\phantom{-}0} =9 \quad \text{and} \quad (28)_{(4)}|^{-1}_{\phantom{-}0} =10.
$$
It follows that there are three possible decompositions of the Hilbert function $\H$, namely 
\begin{equation*}
\begin{array}{ccccccccccccc}
1 & 19 & 17 & 19 & 28  \\
   & 1   & 11 & 11 & 19 \\ \hline 
1 & 18 &   6 &   8 & 9   
\end{array} 
\quad \text{ or } \quad
\begin{array}{ccccccccccccc}
1 & 19 & 17 & 19 & 28  \\
   & 1   & 11 & 11 & 18  \\ \hline 
1 & 18 &   6 &   8 & 10   
\end{array} 
\quad \text{ or } \quad
\begin{array}{ccccccccccccc}
1 & 19 & 17 & 19 & 27  \\
   & 1   & 11 & 11 & 18  \\ \hline 
1 & 18 &   6 &   8 & 9   
\end{array} 
\end{equation*}

\ni The first is eliminated using Lemma \ref{socle}.  The second was already eliminated in part (a) of the proof. 
We thus focus on the third possibility, and we include a consideration of what happens in degree 4. We are either in case (i) or case (ii) above. 

First consider case (i). We know that $J^{\rm sat}$ defines the union of a set of $m \geq 0$ points and a curve, $C$ in $\mathbb P^3$, of degree 9 whose general hyperplane section is the complete intersection of two cubics in the plane. By \cite{strano} or \cite{migliore}, $C$ must be the complete intersection of two cubic surfaces in $\mathbb P^3$. Thus the Hilbert polynomial of $R/J^{\rm sat}$ is $9t-9+m$, so by looking in degree 4 we see $m=0$ and $[J]_4 = [J^{\rm sat}]_4$. Then for a general linear form $\times L : [R/J]_4 \rightarrow [R/J]_5$ is injective, but the same is not true from degree 3 to degree 4, so by Lemma \ref{socle} $R/J$ has socle in degree 3. Then the same is true of $R/I$, and we are done.

Now consider case (ii). The Hilbert polynomial of $R/(J,L_1)$ is $t+5$, so $(J,L_1)$ is the saturated ideal of a line and a complete intersection set of four points in the plane (since it is a quotient of $R/I_\Lambda$), where the complete intersection contains at most a subscheme of degree 2 embedded in the line. 

Now consider $J^{\rm sat}$. In degree 2 it defines the union of a 3-dimensional linear space $\Pi$ and  a set of $m$ points, for some $m \geq 0$. In degree 3 it defines the scheme-theoretic union of  $\leq m$ points and (in $\Pi$) a plane and a curve $C$ of degree 4. Notice that modulo $I_\Lambda$ the ideal $(J,L_1)$ has the form $(LG_1, L G_2)$ where $G_1$ and $G_2$ are a complete intersection (hence independent), so $C$ must be defined by two quadrics, i.e. it must be a complete intersection. So $J^{\rm sat}$ defines a scheme that contains a subscheme (viewed in $\mathbb P^3$) defined by an ideal of the form $(LQ_1, LQ_2)$. Such a subscheme $\X$ already has Hilbert function satisfying ${\bf H}(\X, 4) = 35 - 8 = 27$. Thus $J^{\rm sat}$ is saturated in degree 4, and the same argument that we used for (i) works here.

\end{itemize}

\medskip

\ni
{\be Case 3.} Now consider the last decomposition of $\H$, namely
\begin{equation} \label{EQ:20160602-314}
\begin{array}{ccccccccccccc}
1 & 19 & 17 & 19 & 1  \\
   & 1   & 12 & 12 & 1   \\ \hline 
1 & 18 &   5 &   7 &     
\end{array} 
\end{equation}
Note that, by Gotzmann persistence theorem, the ideal $[(J,L)]$ is $2$-regular and $[(J,L)/(L)]$ defines a conic in a two dimensional linear space in $\L \cong\P^{17}$. This implies that $J^{\sat}$ defines the union of a quadric hypersurface $\mathbb F$ in a $3$-dimensional linear space $\Lambda\subset \P^{18}$ and a finite scheme $\Y$ in $\P^{18}$. Hence we have the following decomposition.
\begin{equation} \label{EQ:20160602-315}
\begin{array}{ccccccccccccc}
1 & 19 & 17 & 19 & 28-\alpha  \\
   & 1   & 12 & 12 & 19+\alpha   \\ \hline 
1 & 18 &   5 &   7 &   9  
\end{array} 
\end{equation}

Since 
$$
\H(R/(J:L_1),3)\le \H(R/J,3)=19,
$$
one can see that $\alpha=0$. Hence we can rewrite equation~\eqref{EQ:20160602-315} as
\begin{equation} \label{EQ:20160602-316}
\begin{array}{ccccccccccccc}
1 & 19 & 17 & 19 & 28 \\
   & 1   & 12 & 12 & 19 \\ \hline 
1 & 18 &   5 &   7 &   9  
\end{array} 
\end{equation}

\ni and by Lemma \ref{socle}, $R/J$ has a $5$-dimensional socle  in degree $2$, hence $R/I$ does as well, which is a contradiction. 
This completes the proof. 
\end{proof} 

As announced at the beginning of this section, we have the following consequence.

\begin{cor} \label{r-2}
A Gorenstein sequence of the form  $(1,r,r-2,r,1)$ exists if and only if $r \geq 20$.
\end{cor}


\section{Acknowledgements}

The first author was supported by the Basic Science Research Program through the National Research Foundation of Korea (NRF), funded by the Ministry of Education, Science, and Technology (No. 2011-0027163); the second author was supported by a grant from the Simons Foundation (grant \#309556); the third author was supported by a grant from Sungshin Women’s University. Part of this paper was written during the Research Station on Commutative Algebra, June 13 -- 18, 2016, which was supported by the Korea Institute of Advanced Study. We thank Mats Boij and Fabrizio Zanello   for helpful comments about the statements in this paper.

\end{document}